\documentclass[12pt,a4paper]{amsart}
\usepackage{amssymb,amsmath}
\usepackage{graphicx}
\usepackage{amssymb}
\usepackage{pstricks}
\usepackage[all,dvips,arc,curve,color,frame]{xy}
\usepackage{calc}
\usepackage{hyperref}
\usepackage{tikz}
\usetikzlibrary{decorations.markings}
\tikzstyle{vertex}=[circle, draw, inner sep=0pt, minimum size=6pt]

\numberwithin{equation}{section}
\newtheorem{theorem}{Theorem}[section]
\newtheorem{lemma}[theorem]{Lemma}
\newtheorem{proposition}[theorem]{Proposition}
\newtheorem{corollary}[theorem]{Corollary}

\theoremstyle{definition}
\newtheorem{definition}[theorem]{Definition}
\newtheorem{definitions}[theorem]{Definitions}
\newtheorem{example}[theorem]{Example}
\newtheorem{examples}[theorem]{Examples}
\newtheorem{remark}[theorem]{Remark}

\newcommand{\vertex}{\node[vertex]}

\newcommand\Ass{\operatorname{Ass}}

\newcommand\Tor{\operatorname{Tor}}
\newcommand\Hom{\operatorname{Hom}}

\newcommand\Ext{\operatorname{Ext}}

\newcommand\reg{\operatorname{reg}}

\newcommand{\projdim}{\operatorname{proj\,dim}}

\begin{document}
\title[On the Betti numbers of some classes of Binomial edge ideals]{On the Betti numbers of some classes of Binomial edge ideals}
\author{Zohaib Zahid, Sohail Zafar}
\address{Abdus Salam School of Mathematical Sciences, GCU, Lahore Pakistan}
\email{Zohaib\_zahid@hotmail.com}

\address{UMT, Lahore Pakistan}
\email{sohailahmad04@gmail.com}
\thanks{This research was partially supported by the Higher Education
Commission, Pakistan}

\begin{abstract}
We study the Betti numbers of binomial edge ideal associated to some classes of graphs with large Castelnuovo-Mumford regularity. As an application we
give several lower bounds of the Castelnuovo-Mumford regularity of arbitrary graphs
depending on induced subgraphs.
\end{abstract}


\maketitle
\section{Introduction}
Let $K$ denote a field. Let $G$ denote a connected, simple and undirected graph over the
vertices labeled by $[n]=\{1,2,\dots,n\}.$ The binomial edge ideal $J_{G}\subseteq
S=K[x_{1},\dots,x_{n},y_{1},\dots,y_{n}]$ is an ideal generated by all binomials $
x_{i}y_{j}-x_{j}y_{i}$ , $i<j$ , such that $\{i,j\}$ is an edge of $G$. It was introduced in \cite{HKR} and independently at the same time in \cite{MO}. It is a natural generalization of the notion of monomial edge ideal which is introduced by Villarreal in \cite{Vi}.

The main purpose of this paper is to study the minimal free resolution of certain classes of binomial edge ideals. The arithmetic properties of binomial edge ideals in terms of combinatorial properties of graphs (and vice versa) were studied by many authors in \cite{HKR}, \cite{MO}, \cite{EHH}, \cite{So}, \cite{ps}, \cite{AR}, \cite{GR}, \cite{MM}, and \cite{SD}. The reduced Gr\H{o}bner basis and minimal primary decomposition of binomial edge ideal was given in the paper of Herzog et al. \cite{HKR}. The Cohen-Macaulay property of binomial edge ideal were studied in \cite{EHH}, \cite{AR} and \cite{GR}. As a certain generalization of the Cohen-Macaulay property the second author has studied approximately Cohen-Macaulay property in \cite{So}.

There is not so much work done so far in the direction of the Betti numbers and Castelnuovo-Mumford regularity of binomial edge ideals. The minimal free resolution of the binomial edge ideal of simplest classes (complete graph and line graph) is well known. In \cite{SD}, the authors
determine the initial Betti number of the binomial edge ideal of an arbitrary graph. In fact they shows that $\beta_{2,1}=2l$ where $l$ is the total number of 3-cycles in graph $G$. They also discussed the vanishing and non-vanishing of few Betti numbers. In \cite{ps}, there is a computation of the Castelnuovo-Mumford regularity and all the Betti numbers in the case of complete bipartite graph. The relationship between the Betti numbers of a graph and the Betti numbers of its induced subgraph (see Theorem \ref{p4}) was recently shown in \cite{MM}. They also gives the Castelnuovo-Mumford regularity bounds for binomial edge ideals, namely $\ell -1 \leq \reg (S/J_G) \leq n-1,$ where $\ell$ denotes the number of vertices of the longest induced line graph of $G$.

In the present paper we compute Castelnuovo-Mumford regularity and all the Betti numbers in the case of cycle graph and two more classes of graphs $\mathcal{T}_3$ and $\mathcal{G}_3$ (see Definitions \ref{d1} and \ref{d2}). In all of our classes Castelnuovo-Mumford regularity is quite large.  As an application of our investigation we improve the lower bounds for the Castelnuovo-Mumford regularity of an arbitrary graph by applying as it was done in \cite{MM}.

The paper is organized as follows:
In Section 2, we introduced some notations and give some results that we need in the rest of the paper. In particular we give a short summary on minimal free resolutions. In Section 3, we compute Castelnuovo-Mumford regularity and Betti numbers of the binomial edge ideal associated of a cycle graph and obtained a lower bound for the Castelnuovo-Mumford regularity of an arbitrary graph. In Section 4, we do the same for the classes of graphs $\mathcal{T}_3$ and $\mathcal{G}_3$ as we did for the cycle in Section 3.

\section{Preliminaries}
In this section we will introduce the notation used in the article. Moreover we
summarize a few auxiliary results that we need.

We denote by $G$ a connected undirected graph on $n$ vertices labeled by
$[n] = \{1,2,\ldots,n\}$. For an arbitrary field $K$ let $S = K[x_1,\dots,x_n,y_1,\dots,y_n]$
denote the polynomial ring in the $2n$ variables. To the
graph $G$ one can associate an ideal $J_G \subset S$ generated by all binomials
$x_iy_j-x_jy_i$ for $i<j$ such that $\{i,j\}$ forms an edge of $G$. This Ideal $J_G$ is called \textbf{binomial edge ideal} associated to the graph $G$.
This construction was invented by Herzog et al. in \cite{HKR} and independently found in \cite{MO}. At first let us
recall some of their definitions.

\begin{definition} \label{p1} Fix the previous notation. For a set $T \subset [n]$
let $\tilde{G}_T$ denote the complete graph on the vertex set $T$. Moreover
let $G_{[n]\setminus{T}}$ denote the graph obtained by deleting all vertices
of $G$ that belong to $T$.

Let $c = c(T)$ denote the number of connected components of $G_{[n]\setminus{T}}$.
Let $G_1,\ldots,G_c$ denote the connected components of $G_{[n]\setminus{T}}$. Then define
\[
P_{T}(G)=(\cup _{i\in T}\{x_{i},y_{i}\},J_{\tilde{G}_{1}},\dots,J_{\tilde{G}%
_{C(T)}}),
\]
where $\tilde{G}_i, i =1,\ldots,c,$ denotes the complete graph on the vertex set
of the connected component $G_i, i = 1,\ldots,c$.
\end{definition}

The following result is important for the understanding of the binomial edge ideal
of $G$.

\begin{lemma} \label{p2}
With the previous notation the following holds:
\begin{itemize}
\item[(a)] $P_{T}(G)\subset S$
is a prime ideal of height $n-c+|T|,$ where $|T|$ denotes the number of
elements of $T$.
\item[(b)] $J_{G}=\cap _{T \subseteq [n]}P_{T}(G).$
\item[(c)] $J_{G}\subset P_{T}(G)$ is a
minimal prime if and only if either  $T=\emptyset $ \ or $T\neq \emptyset $
and $c(T \setminus \{i\})<c(T)$  for each $i \in T$.
\end{itemize}
\end{lemma}

\begin{proof}
For the proof we refer to Section 3 of the paper \cite{HKR}.
\end{proof}

Therefore $J_{G}$ is the intersection of prime ideals. That is, $S/J_{G}$ is
a reduced ring. Moreover, we remark that $J_G$ is an ideal generated by quadrics
and therefore homogeneous, so that $S/J_{G}$ is a graded ring with natural grading induced
by the $\mathbb{N}$-grading of $S$.
\begin{remark} If we define a grading on $S$ by setting $\deg x_i=\deg y_i=e_i$ where $e_i$ is the $i-$th unit vector of $\mathbb{N}^n$ then it is easy to see that $S/J_{G}$ is $\mathbb{N}^n$-graded too.
\end{remark}
Let $M$ a graded finitely generated $S$-module. By Hilbert syzygy Theorem, $M$ has a finite minimal graded free resolution:
\[
F_{\bullet}: 0\to F_p\to \cdots \to F_1\to F_0\to M \to 0
\]
where $F_i=\bigoplus_jS(-d_{ij})^{\beta_{ij}}$ for $i\geq 0$ and $p$ is called the \textbf{projective dimension} of $M$.
The numbers $\beta_{ij}$ are uniquely determined by $M$ i.e. $\beta_{i,j}(M) = \dim_K \Tor_i^S(K,M)_{i+j}, i,j \in \mathbb{Z},$ as \textbf{graded Betti numbers} of $M$. We can also define \textbf{Castelnuovo-Mumford regularity} $\reg M = \max \{j \in \mathbb{Z} | \beta_{i,j}(M) \not= 0\}$. The \textbf{Betti table} looks as in the following:
\[
\begin{array}{c|ccccc}
   & 0 & 1  & \cdots & p \\
\hline
0 & \beta_{0,0} & \beta_{1,0} & \cdots & \beta_{p,0} \\
1 & \beta_{0,1} & \beta_{1,1} & \cdots & \beta_{p,1} \\
\vdots & \vdots & \vdots & \vdots \\
r & \beta_{0,r} & \beta_{1,r} & \cdots & \beta_{p,r}
\end{array}
\] Note that all the $\beta_{i,j}$ outside
of the Betti table are zero. For more details and related facts we refer the book of Burns and Herzog \cite{BH}.
\begin{definition} \label{p3} Let $M$ denote a finitely generated graded
$S$-module and $d = \dim M$. For an integer $i \in \mathbb{Z}$ put
\[
\omega^i(M) = \Ext_S^{2n-i}(M,S(-2n))
\]
and call it the \textbf{$i$-th module of deficiency}. Moreover we define $\omega(M) = \omega^d(M)$
the \textbf{canonical module} of $M$.
\end{definition}
These modules have been introduced and studied
in \cite{Sch1}.
\begin{theorem}\label{p5} Let $M$ denote a finitely generated graded Cohen-Macaulay $S$-module of projective dimension $p = 2n - \dim M$. Let
\[
F_{\bullet}: 0\to F_p\to \cdots \to F_1\to F_0\to 0
\] be the minimal free resolution of $M$. Let $G_{\bullet}=\Hom_S(F_{\bullet},S(-2n))$ be the dual complex
\[
G_{\bullet}: 0\to G_p\to \cdots \to G_1\to G_0\to 0
\] where $G_i=F_{p-i}$ for i=0,\ldots,p. Then $G_{\bullet}$ is the minimal free resolution of $\omega(M).$
\end{theorem}
\begin{proof}
For the proof see \cite[Corollary 3.3.9]{BH}.
\end{proof}
For a subset $W\subset[n]$, a graph $G_W$ on vertex set $W$ is called induced subgraph of $G$ if for all $i,j\in W$, $\{i,j\}$ is an edge of $G_W$ if and only if $\{i,j\}$ is an edge of $G$. Recently K. Matsuda and S. Murai in \cite{MM} proved the relationship between the Betti numbers of the graph with the Betti numbers of its induced subgraph. The result is as follows:
\begin{theorem}\label{p4}
Let $G_W$ be the induced subgraph of $G$. Then $\beta_{i,j}(S/J_G)\geq \beta_{i,j}(S/J_{G_W})$ for all $i,j$.
\end{theorem}
\begin{proof}
For the proof see \cite[Corollary 2.2]{MM}.
\end{proof}
\section{Betti Numbers of the binomial edge ideal of a cycle}
\begin{definition}
A \textbf{cycle} is a graph in which all the vertices are of degree 2.
\end{definition}
In particular, for $n=3$ it is triangle and $n=4$ it is square.
We denote the cycle on vertex set $[n]$ by $C$ and its binomial edge ideal by $I_{C}$. It is known from \cite[Theorem 4.5]{So} that, $S/I_{C}$ is approximately Cohen-Macaulay ring of $\dim(S/I_C)=n+1$.
\begin{theorem}\label{c1}Let $I_{L}$ be the binomial edge ideal of a
line $L$ on vertex set $[n]$ , $g=x_{1}y_{n}-x_{n}y_{1}$ and $J_{\tilde{G}}$
be binomial edge ideal of a complete graph on $[n]$ then:
\begin{enumerate}
  \item [(a)] $I_L:g/I_L\cong \omega (S/J_{\tilde{G}})(2)$
  \item [(b)] The Hilbert series of $S/I_C$ is \[H(S/I_{C},t)=\frac{1}{(1-t)^{n+1}}((1+t)^{n-1}-t^{2}(1+t)^{n-1}+(n-1)t^{n}+t^{n+1}).\]
\end{enumerate}

\end{theorem}
\begin{proof} For the proof see \cite[Lemma 4.8 and Theorem 4.10]{So}.
\end{proof}

In order to compute the Betti numbers of a cycle we need to understand the modules $I_L:g/I_L$ and $S/I_L:g$. To this end we need the following lemma about the canonical module of $S/I_L:g$.
\begin{lemma} \label{cc1} With the notation before we have,
\begin{itemize}
 \item[(a)] $\omega (S/I_{L}:g)\cong J_{\tilde{G}}/I_{L}(-2).$
 \item[(b)] Minimal number of generators of $\omega (S/I_{L}:g)$ is $\binom{n-1}{2}$.
\end{itemize}
\end{lemma}
\begin{proof} From Theorem \ref{c1} we have $I_L:g/I_L\cong \omega (S/J_{\tilde{G}})(2)$. Now consider the exact sequence
\[0\to \omega (S/J_{\tilde{G}})(2)\to S/I_L\to S/I_L:g\to 0.\]
All modules in above exact sequence are Cohen-Macaulay of dimension $n+1$. By applying local cohomology and dualizing it we get the following exact sequence
\[0\to \omega (S/I_L:g )\to S/I_L(-2)\to S/J_{\tilde{G}}(-2)\to 0.\] Which implies the isomorphism in (a) and then (a) gives us (b).
\end{proof}
All $\Tor$ modules of $I_{L}:g/I_L$ are given in the following lemma.
\begin{lemma} We have the following isomorphisms.
\begin{itemize}
 \item[(a)] $\Tor_i^S(K,I_{L}:g/I_L) \cong K^{c_i}(-n+2-i)$ for $i = 0,\ldots,n-2,$ where $c_i = (n-1-i) \binom{n}{i}$.
 \item[(b)] $\Tor_{n-1}^S(K,I_{L}:g/I_L) \cong K(-2n+2).$

\end{itemize}
\end{lemma}
\begin{proof} It is well known that $S/J_{\tilde{G}}$ is Cohen-Macaulay with the minimal free
resolution
\[0 \to S^{b_{n-1}}(-n) \to \cdots \to S^{b_{n-1-i}}(-n+i) \to \cdots \to S^{b_1}(-2) \to S\] where $b_i=i \binom{n}{i+1}$. By Lemma \ref{c1} (a) and Theorem \ref{p5}, we have the above statement.
\end{proof}
Now we will give the Theorem in which we compute all $\Tor$ modules of $S/I_L:g.$
\\
\\
\begin{theorem} With the previous notation we have
\begin{itemize}
 \item[(a)] $\Tor_i^S(K,S/I_L:g) \cong K^{\binom{n-1}{i}}(-2i)\oplus K^{c_{i-1}}(-n+3-i)$ for $i = 1,\ldots,n-3,$
 \item[(b)] $\Tor_{n-2}^S(K,S/I_L:g) \cong K^{c_{n-3}}(-2n+5),$
 \item[(c)] $\Tor_{n-1}^S(K,S/I_L:g) \cong K^{\binom{n-1}{2}}(-2n+4),$
 \item[(d)] $\reg(S/I_L:g)=n-3.$
\end{itemize}
\end{theorem}
\begin{proof} Consider the exact sequence \[0\to I_L:g/I_L \to S/I_L\to S/I_L:g\to 0.\]
Let $i<n-2$, then the above exact sequence induces a graded homomorphism of degree zero
\[
\Tor_i^S(K,I_L:g/I_L)\cong K^{c_i}(-n+2-i) \to \Tor_i^S(K,S/I_L) \cong K^{\binom{n-1}{i}}(-2i).
\]
Therefore it is the zero homomorphism so we have the following isomorphism
\[
\Tor_i^S(K,S/I_L:g) \cong \Tor_i^S(K,S/I_L) \oplus \Tor_{i-1}^S(K,I_L:g/I_L).
\]
Let $i=n-1$, then we get the injection $0 \to K(-2n+2) \to K(-2n+2)$ which is in fact
an isomorphism. So we have the following exact sequence of K-vector spaces.
\begin{gather*}
0 \to \Tor_{n-1}^S(K,S/I_L:g) \to \Tor_{n-2}^S(K,I_L:g/I_L)\cong K^{c_{n-2}}(-2n+4) \to \\ \Tor_{n-2}^S(K,S/I_L) \cong K^{\binom{n-1}{1}}(-2n+4)\to
\Tor_{n-2}^S(K,S/I_L:g) \to \Tor_{n-3}^S(K,I_L:g/I_L)\to 0.
\end{gather*}
By Lemma \ref{cc1} (b) and Theorem \ref{p5} we have $\Tor_{n-1}^S(K,S/I_L:g) \cong K^{\binom{n-1}{2}}(-2n+4)$ since $S/I_L:g$ is Cohen-Macaulay. By investigating the K-vector space dimension of these modules and $c_{n-2}=\binom{n}{2}=\binom{n-1}{2}+\binom{n-1}{1}$, it follows that $\Tor_{n-2}^S(K,S/I_L:g) \cong \Tor_{n-3}^S(K,I_L:g/I_L).$
\end{proof}
\begin{lemma}\label{cc3} The coefficient of the highest power $t^{n-1}$ of the numerator of
the Hilbert series $H(S/I_C,t)$ is $\binom{n-1}{2}-1$.
\end{lemma}
\begin{proof} If we expand $(1+t)^{n-1}$ in the numerator of $H(S/I_C,t)$ of Theorem \ref{c1} (b). The last two terms in the numerator cancels and we get $1+(n-1)t+ \cdots +(\binom{n-1}{2}-1)t^{n-1}$.
\end{proof}
Now we are ready to say about all $\Tor$ modules of $S/I_C$.
\begin{theorem}\label{cc2} With the previous notation we have
\begin{itemize}
 \item[(a)] $\Tor_i^S(K,S/I_C) \cong K^{\binom{n}{i}}(-2i)\oplus K^{c_{i-2}}(-n+2-i)$ for $i = 1,\ldots,n-2,$
 \item[(b)] $\Tor_{n-1}^S(K,S/I_C) \cong K^{c_{n-3}}(-2n+3),$
 \item[(c)] $\Tor_{n}^S(K,S/I_C) \cong K^{\binom{n-1}{2}-1}(-2n+2).$

\end{itemize}
\end{theorem}
\begin{proof} Consider the exact sequence
\[0\rightarrow S/I_{L}:g(-2) \rightarrow S/I_{L}\rightarrow S/I_{C}\rightarrow 0.\]
Let $i<n-1$, then the above exact sequence induces a graded homomorphism of degree zero
\[
\Tor_i^S(K,S/I_L:g)(-2)\cong K^{\binom{n-1}{i}}(-2i-2)\oplus K^{c_{i-1}}(-n+1-i) \to \Tor_i^S(K,S/I_L) \cong K^{\binom{n-1}{i}}(-2i).
\]
Therefore it is the zero homomorphism so we have the following isomorphism
\[
\Tor_i^S(K,S/I_C) \cong \Tor_i^S(K,S/I_L) \oplus \Tor_{i-1}^S(K,S/I_L:g)(-2).
\]
Let $i=n$, then we have the following exact sequence of K-vector spaces
\begin{gather*}
0 \to \Tor_{n}^S(K,S/I_C) \to \Tor_{n-1}^S(K,S/I_L:g)(-2)\cong K^{\binom{n-1}{2}}(-2n+2) \to \Tor_{n-1}^S(K,S/I_L)\cong \\K(-2n+2) \to  \Tor_{n-1}^S(K,S/I_C) \to \Tor_{n-2}^S(K,S/I_L:g)(-2) \cong K^{c_{n-3}}(-2n+3)\to 0.
\end{gather*}
It is clear from Lemma \ref{cc3}  that $\Tor_{n}^S(K,S/I_C)\cong K^{\binom{n-1}{2}-1}(-2n+2)$ which further implies that $\Tor_{n-1}^S(K,S/I_C) \cong K^{c_{n-3}}(-2n+3).$

\end{proof}
 As a final result on the binomial edge ideal of a cycle we describe the explicit values of Betti numbers in the next corollary.
\begin{corollary}\label{cc4} Let $S/I_C$ be binomial edge ideal of cycle on vertex set $[n]$. Then
\begin{enumerate}
  \item [(a)]$\reg(S/I_C)=n-2,$
  \item [(b)]We have the following non zero Betti numbers for $S/I_C$ on the diagonal of the Betti diagram
$$\beta_{i,j}=\binom{n}{i}\text{, if  } i=j=0,\ldots,n-3$$
and the last row of Betti diagram
\begin{eqnarray*}
\beta_{i,n-2}=\left\{\begin{array}{ll}
c_{i-2}, & \hbox{if \, $i = 2,\ldots,n-3$ ;}\\
\binom{n}{2}+ c_{n-4}, & \hbox{if \, $i=n-2$ ;}\\
 c_{n-3}, & \hbox{if \, $i=n-1$ ;}\\
 \binom{n-1}{2}-1, & \hbox{if \, $i=n.$}
\end{array}\right.
\end{eqnarray*}
\end{enumerate}

\end{corollary}
\begin{proof} It follows from Theorem \ref{cc2}.
\end{proof}
\begin{corollary} \label{xx} Let G be any arbitrary graph on vertex set $[n]$. Let $C$ denote a cycle on maximal
$k$ vertices as an induced subgraph. Then $\reg(S/J_G)\geq k-2$ and $\beta_{i,j}(S/J_G) \geq \beta_{i,j}(S/I_C)$, where the values of $\beta_{i,j}(S/I_C)$ are those of Corollary \ref{cc4}
for $n = k$..
\end{corollary}
\begin{proof} It follows from Theorem \ref{p4} and Corollary \ref{cc4}.
\end{proof}

\begin{remark} In case $G$ has a cycle $C$ on maximal $k$ vertices as an induced subgraph it has also
a line $L$ on $k-1$ vertices. That is, the lower bound  $k-2 \leq \reg (S/J_G)$ is not better than those
of \cite{MM}. The advantage of Corollary \ref{xx} is that it provides the non-vanishing of certain
Betti numbers different from those of $\beta_{i,j}(S/I_L)$.

\end{remark}

\section{Betti Numbers of the binomial edge ideal of $\mathcal{T}_3$ and $\mathcal{G}_3$ }
The clique complex was used by many authors (see e.g. \cite{EHH},\cite{AR} and \cite{GR}) to prove several results in the case of binomial edge ideals. Here in the following we introduce this nice concept.
\begin{definitions}Let $G$ be a simple graph on vertex set $[n]$.

\begin{enumerate}
                     \item A \textbf{clique} of $G$ is a subset $W$ of $[n]$
such that each vertex in $W$ is connected to any other vertex in $W$ by an edge
of $G$. In other words it is a complete subgraph of $G$.
\item A \textbf{maximal clique} is a clique that is not a subset of a larger clique.
                     \item The \textbf{clique complex} $\Delta(G)$ of $G$ is a simplicial complex whose facets are the maximal cliques of $G$.
                     \item A vertex $j\in[n]$ is called \textbf{free vertex} if it belongs to only one facet of $\Delta(G)$.
                   \end{enumerate}
\end{definitions}

\begin{examples}There are some examples of free vertices.
\begin{enumerate}
                  \item In complete graph all the vertices are free vertices.
                  \item In any graph the vertices of degree $1$ are free vertices.
                  \item A cycle graph of vertices more then $4$ has no free vertex.
                \end{enumerate}
\end{examples}
 The following Proposition in \cite{AR} is important for us.
\begin{proposition}\label{acc} Let $G$ be a simple graph on vertex set $[n]$. Let $\Delta(G)$ is clique complex of $G$ and $j\in[n]$ be a vertex of $G$. Then the following conditions are equivalent:
\begin{enumerate}
                  \item $j$ is a free vertex of $\Delta(G)$.
                  \item $j \notin T$ for all $T\subseteq[n]$ such that $c(T \setminus \{i\})<c(T)$  for each $i \in T$.

                \end{enumerate}
\end{proposition}
\begin{proof} For the proof see \cite[Proposition 2.1]{AR}.
\end{proof}
The following lemma tells us the importance of the free vertex.
\begin{lemma}\label{b1}
 Let $G$ be the graph on vertex set $[n]$ with at least one free vertex and $J_{G}$ be its binomial edge ideal. Chose one of its free vertex and label it by $n$. Let $J_{G^\prime}$ denotes the binomial edge ideal of the graph $G^\prime$ by attaching $\{n, n+1\}$ to the graph G. Then $f=x_ny_{n+1}-x_{n+1}y_n$ is regular on $S^\prime/J_{G}$ where $S^\prime=S[x_{n+1},y_{n+1}]$ and we have the following exact sequence of $S^\prime$ modules
 \[0\rightarrow S^\prime /J_{G}(-2)\mathop\rightarrow\limits^f S^\prime /J_{G}\rightarrow S^\prime /J_{G^\prime}\rightarrow 0.\]
where $J_{G^\prime}=(J_{G},f)$
\end{lemma}
\begin{proof}Since $n$ is free vertex therefore by Proposition \ref{acc} and Lemma \ref{p2} (c) $x_{n},y_{n}$ $\notin P_{T}(G)$
for all $\ P_{T}(G)\in \Ass(S/J_{G}),$ and hence $f$ $\notin P_{T}(G)$
for all $\ P_{T}(G)\in \Ass(S^\prime/J_{G}),$ therefore $f$ is not a zero divisor in $S^\prime/J_{G}$ and is regular. Above exact sequence is easily seen because of the regularity of $f$ on $S^\prime/J_{G}$.
\end{proof}
\begin{definition}\label{d1}$\mathcal{T}_3$ be the collection of graphs such that for all $G\in\mathcal{T}_3$ we have
\[V (G) = \{u_1, \ldots , u_r, v_1, \ldots , v_s,w_1, \ldots ,w_t\}\]
with $r \geq 2, s \geq 1, t \geq 1$ and edge set
\begin{gather*}E(G) =\{\{u_i, u_{i+1}\} : i = 1, \ldots , r - 1\} \cup \{\{v_i, v_{i+1}\} : i = 1, \dots , s - 1\}\\ \cup
�\{\{w_i,w_{i+1}\} : i = 1, \ldots , t - 1\} �\cup \{\{u_1, v_1\}, \{u_1,w_1\}\}.
\end{gather*}

\end{definition}
Note that any $G\in\mathcal{T}_3$ is a tree.
\[\begin{tikzpicture}
	
	\vertex[fill] (1) at (0,0) [label=below:$u_2$] {};
	\vertex[fill] (2) at (1.5,0) [label=below:$u_1$] {};
	\vertex[fill] (4) at (2.5,1) [label=above:$v_1$] {};
	\vertex[fill] (3) at (2.5,-1) [label=below:$w_1$] {};
	\path
		(1) edge (2)
		(3) edge (2)
		(2) edge (4)

	;
\end{tikzpicture}\]
\begin{center}
FIGURE \large{1}
\end{center}

\begin{example}\label{b2}

Consider the simplest example of the graph in $\mathcal{T}_3$ as shown in
figure 1. It is easy to see that $u_2$, $v_1$ and $w_1$ are free vertices of $G$. Its binomial edge ideal has the following Betti diagram
\[
\begin{array}{c|cccc}
   & 0 & 1 & 2 &  3 \\
\hline
0 & 1 & 0 & 0  & 0 \\
1 & 0 & 3 & 0 &  0 \\
2 & 0 & 0 & 4 &  2
\end{array}
\]
and it has the following Hilbert series
\[H(S/J_{G},t)=\frac{1}{(1-t)^{6}}(1+2t-2t^{3}).\]

\end{example}
It is known from \cite[Corollary 3.5]{So} that any $G\in\mathcal{T}_3$ on vertex set $[n]$ is approximately Cohen-Macaulay ring of $\dim(S/J_G)=n+2$.
\begin{remark} We use computer algebra system CoCoA \cite{CO} for the computations of some arithmetic invariants of $S/J_G$ in Example \ref{b2} and \ref{b3}.
\end{remark}
\begin{definition}\label{d2} (see \cite{GR}) $\mathcal{G}_3$ be the collection of graphs such that for all $G\in\mathcal{G}_3$ we have
\[V (G) = \{u_1, \ldots , u_r, v_1, \ldots , v_s,w_1, \ldots ,w_t\}\]
with $r \geq 1, s \geq 1, t \geq 1$ and edge set
\begin{gather*}E(G) =\{\{u_i, u_{i+1}\} : i = 1, \ldots , r - 1\} \cup \{\{v_i, v_{i+1}\} : i = 1, \dots , s - 1\}\\ \cup
�\{\{w_i,w_{i+1}\} : i = 1, \ldots , t - 1\} �\cup \{\{u_1, v_1\}, \{u_1,w_1\}, \{v_1,w_1\}\}.
\end{gather*}
\end{definition}
\begin{example}\label{b3}

The simplest example of the graph in $\mathcal{G}_3$ is complete graph on vertices $u_1$, $v_1$ and $w_1$ and all of them are free vertices of $G$. Its binomial edge ideal has the following Betti diagram
\[
\begin{array}{c|ccc}
   & 0 & 1 & 2  \\
\hline
0 & 1 & 0 & 0  \\
1 & 0 & 3 & 2

\end{array}
\]
and it has the following Hilbert series \[H(S/J_G,t)=\frac{1}{(1-t)^4}(1+2t).\]

\end{example}
It is known from \cite[Proposition 2.5]{GR} that any $G\in\mathcal{G}_3$ on vertex set $[n]$ is Cohen-Macaulay ring of $\dim(S/J_G)=n+1.$ Note that the $\projdim(S/J_G)=n-1$ for any $G\in\mathcal{T}_3\cup\mathcal{G}_3$ on vertex set $[n]$ as it will be shown in the following result.

\begin{theorem}\label{b4} Let $G$ be the graph on vertex set $[n]$ and $G\in\mathcal{T}_3\cup\mathcal{G}_3$. Let $J_G$ denotes its binomial edge ideal then the $\reg(S/J_G)=n-2$ and the Betti diagram of the $S/J_G$ looks like the following:

\[
\begin{array}{c|ccccccccc}

   & 0 & 1  & 2 &  3  &  4  & \cdots & n-2 & n-1  \\
\hline
0 & 1 & 0 & 0 & 0 & 0 & \cdots & 0& 0  \\
1 & 0 & \beta_{1,1} & \beta_{2,1} & 0 & 0 & \cdots & 0 & 0  \\
2 & 0 & 0 & \beta_{2,2} & \beta_{3,2} & 0 & \cdots & 0 & 0  \\
3 & 0 & 0 & 0 & \beta_{3,3} & \beta_{4,3} & \cdots & 0 & 0  \\
\vdots & \vdots & \vdots & \vdots & \vdots & \ddots & \ddots
& \vdots & \vdots\\
n-3 & 0 & 0 & 0 & 0& 0 & \ddots & \ddots & 0  \\
n-2 & 0 & 0 & 0 & 0 & 0 & \cdots & \beta_{n-2,n-2}& \beta_{n-1,n-2}
\end{array}
\]

\end{theorem}
\begin{proof}We want to prove the claim on the Betti table by induction on $n$. For $n=3$ or $n=4$ it is true (see Example \ref{b2} and \ref{b3}). Now  let us assume that the statement is true for $n$. We use the notations of Lemma \ref{b1} because in these classes we inductively go from graph $G$ on $n$ vertices to graph $G^\prime$ on $n+1$ vertices by adding an edge $\{n,n+1\}$ with the assumption that $n$ is free vertex. Let $F_{\bullet}$ be the minimal free resolution for $S/J_G$ then $F_{\bullet}\otimes_S S^\prime$ is minimal free resolution for $S^\prime /J_{G}$ and hence they have the same Betti numbers. To this end we note that
\[\Tor_i^S(K,S /J_{G})\cong \Tor_i^{S^\prime}(K,S^\prime /J_{G})\]
Hence we have the following isomorphism restricted to degree $i+j$.
\[\Tor_i^{S^\prime}(K,S^\prime /J_{G})_{i+j}\cong \Tor_i^S(K,S /J_{G})_{i+j} \cong K^{\beta_{i,j}}(-(i+j)) \]
Now consider the exact sequence of Lemma \ref{b1}
\[0\rightarrow S^\prime /J_{G}(-2)\mathop\rightarrow\limits^f S^\prime /J_{G}\rightarrow S^\prime /J_{G^\prime}\rightarrow 0.\]
The first two modules of the above exact sequence are the same modules with the shift of degree $2$. Next we want to show that the map
\[
\phi_{i}: \Tor_i^{S^\prime}(K,S^\prime /J_{G}(-2)) \to \Tor_i^{S^\prime}(K,S^\prime /J_{G}).
\]
is the zero map. To this end it will be enough to show that
\[
[\phi_{i}]_{j}: \Tor_i^{S^\prime}(K,S^\prime /J_{G})_{i+j-2} \to \Tor_i^{S^\prime}(K,S^\prime /J_{G})_{i+j}.
\]
is zero for all $j$. Now suppose that
\[
0 \neq \Tor_i^{S^\prime}(K,S^\prime /J_{G})_{i+j-2} \cong \Tor_i^{S}(K,S /J_{G})_{i+j-2}.
\]
By induction hypothesis for $n$ it turns out that $(i,j-2)$ is either $(i,i)$ or $(i,i-1)$. In the first case, that is $j-2=i$, the target of $[\phi_{i}]_{j}$ is
\[
\Tor_i^{S^\prime}(K,S^\prime /J_{G})_{2i+2} \cong \Tor_i^{S}(K,S /J_{G})_{2i+2}=0.
\]
In the second case, that is $j-2=i-1$, the target of $[\phi_{i}]_{j}$ is
\[
\Tor_i^{S^\prime}(K,S^\prime /J_{G})_{2i+1}\cong \Tor_i^{S}(K,S /J_{G})_{2i+1}=0.
\]
Now suppose that the target of $[\phi_{i}]_{j}$ namely $\Tor_i^{S^\prime}(K,S^\prime /J_{G})_{i+j}$ is non-zero. Again by induction hypothesis for $n$ it follows that $(i,j)$ is either $(i,i)$ or $(i,i-1)$. In the first case, that is $j=i$, the source of $[\phi_{i}]_{j}$ is
\[
\Tor_i^{S^\prime}(K,S^\prime /J_{G})_{2i-2} \cong \Tor_i^{S}(K,S /J_{G})_{2i-2}=0.
\]
In the second case, that is $j=i-1$, the source of $[\phi_{i}]_{j}$ is
\[
\Tor_i^{S^\prime}(K,S^\prime /J_{G})_{2i-3} \cong \Tor_i^{S}(K,S /J_{G})_{2i-3}=0.
\]
This completes the proof for $\phi_{i}$ is the zero map. Therefore the short exact sequence induces an isomorphism
\[
\Tor_i^{S^\prime}(K,S^\prime /J_{G^\prime}) \cong \Tor_i^{S^\prime}(K,S^\prime /J_{G}) \oplus \Tor_{i-1}^{S^\prime}(K,S^\prime /J_{G}(-2)).
\]
In order to complete the inductive step we have to show that $\Tor_i^{S^\prime}(K,S^\prime /J_{G^\prime})_{i+j}$ is zero for all $(i,j)$ different of $(i,i)$ and $(i,i-1)$. This follows because of
\[
\Tor_i^{S^\prime}(K,S^\prime /J_{G^\prime})_{i+j} \cong \Tor_i^{S^\prime}(K,S^\prime /J_{G})_{i+j} \oplus \Tor_{i-1}^{S^\prime}(K,S^\prime /J_{G})_{i+j-2}.
\]
Note that if $(i,j)\neq(i,i-1)$ and $(i,j)\neq(i,i)$, then $(i-1,j-1)\neq(i-1,i-2)$ and $(i-1,j-1)\neq(i-1,i-1)$.
Moreover we get
\[
\beta_{i,i}(S^\prime /J_{G^\prime}) = \beta_{i,i}(S/J_{G})+ \beta_{i-1,i-1}(S/J_{G}).
\]
and
\[
\beta_{i,i-1}(S^\prime /J_{G^\prime}) = \beta_{i,i-1}(S/J_{G})+ \beta_{i-1,i-2}(S/J_{G}).
\]
Hence any Betti number of $S^\prime /J_{G^\prime}$ is the sum of consecutive Betti numbers of $S/J_{G}$ of its same diagonal. This completes the inductive step.
\end{proof}
The recursion formulas for the Betti numbers at the end of the proof might be used for an explicit computation of them. We will follow here a different approach using Hilbert functions.
\begin{lemma}\label{b5}
    With the notations of Lemma \ref{b1}, we have $$H(S^\prime /J_{G^\prime},t)=(1-t^{2})H(S^\prime /J_{G},t).$$
\end{lemma}
\begin{proof}
Consider the exact sequence \[0\rightarrow S^\prime /J_{G}(-2)\mathop\rightarrow\limits^f S^\prime /J_{G}\rightarrow S^\prime /J_{G^\prime}\rightarrow 0.\]Hence we have a required result.
\end{proof}
\begin{lemma}\label{b6} Let $G$ be a graph on vertex set $[n]$ and $J_G$ be its binomial edge ideal.
\begin{enumerate}
  \item[(a)] Let $G\in\mathcal{T}_3$ then the Hilbert series is
  \[ H(S/J_{G},t)=\frac{1}{(1-t)^{n+2}}(1+2t-2t^{3})(1+t)^{n-4}\text{ \ for }n>3.\]
  \item[(b)] Let $G\in\mathcal{G}_3$ then the Hilbert series is
  \[ H(S/J_{G},t)=\frac{1}{(1-t)^{n+1}}(1+2t)(1+t)^{n-3}\text{ \ for }n>2.\]
\end{enumerate}
\end{lemma}
\begin{proof}
We will prove (a) by induction on $n$. For $n=4$ it is true, see Example \ref{b2}. Suppose the claim is true for $n.$ That is,
 $$H(S^\prime /J_{G},t)=\frac{1}{(1-t)^{n+4}}(1+2t-2t^{3})(1+t)^{n-4}.$$ Now by previous lemma we have
\begin{equation*}
H(S^\prime /J_{G^\prime},t)=\frac{1}{(1-t)^{n+3}}(1+2t-2t^{3})(1+t)^{n-3}
\end{equation*}%
as required.

Similar arguments might be used in order to calculate the Hilbert series
in (b).
\end{proof}
\begin{theorem}\label{b7} Let $G$ be a graph on vertex set $[n]$ and $J_G$ be its binomial edge ideal.
\begin{enumerate}
  \item[(a)] Let $G\in\mathcal{T}_3$ then the Betti numbers for $S/J_G$ are:
  \begin{eqnarray*}
\beta_{i,j}=\left\{\begin{array}{ll}
\binom{n-4}{i}+3\binom{n-4}{i-1}+4\binom{n-4}{i-2}, & \hbox{if \, $j = i=0,\ldots,n-2$ ;}\\
2\binom{n-4}{i-3}, & \hbox{if \, $j=i-1=2,\ldots,n-2$ ;}\\
 0, & \hbox{ \, otherwise .}

\end{array}\right.
\end{eqnarray*}
  \item[(b)] Let $G\in\mathcal{G}_3$ then the Betti numbers for $S/J_G$ are:
  \begin{eqnarray*}
\beta_{i,j}=\left\{\begin{array}{ll}
3\binom{n-3}{i-1}+\binom{n-3}{i}, & \hbox{if \, $j = i=0,\ldots,n-2$ ;}\\
2\binom{n-3}{i-2}, & \hbox{if \, $j=i-1=1,\ldots,n-2$ ;}\\
 0, & \hbox{ \, otherwise .}

\end{array}\right.
\end{eqnarray*}
  \end{enumerate}
\end{theorem}
\begin{proof}
By the additivity of the Hilbert series and the structure of the Betti table (see in the above Theorem \ref{b4}) provides the following formula
\[
H(S/J_G,t) = \frac{1}{(1-t)^{2n}}\left(\sum_{i=0}^{n-2} (-1)^i\beta_{i,i} t^{2i}+\sum_{i=2}^{n-1} (-1)^{i}\beta_{i,i-1} t^{2i-1}\right).
\]
and after comparing the Hilbert series of Lemma \ref{b6} (a) and by making a few simple computations we have the formulas for the Betti numbers in (a). Similar computation using the Hilbert series of Lemma \ref{b6} (b) gives the Betti numbers in (b).
\end{proof}
\begin{corollary}\label{xy} Let G be any arbitrary graph on vertex set $[n]$. Suppose that $G$ has an induced subgraph $H\in\mathcal{T}_3\cup\mathcal{G}_3$ on $k$ vertices. Then $\reg(S/J_G)\geq k-2$.
\end{corollary}
\begin{proof} It is an easy consequence of Theorem \ref{p4} and \ref{b4}.
\end{proof}

\begin{remark}  Let $G$ denote a graph with the largest $k$ such that $G$ has an induced
subgraph $H \in\mathcal{T}_3\cup\mathcal{G}_3$. Then it has also a line $L$ as an induced subgraph
with $\ell = \max \{s+t+1,r+t,r+s\}$ resp. $\ell = \max \{r+s,r+t,s+t\}$ vertices.
In general $k > \ell$, so
that the lower bound for the Castelnuovo-Mumford regularity in Corollary \ref{xy} improves those of \cite{MM}.
\end{remark}

\end{document}